\documentclass[10pt,twoside]{article}

\RequirePackage{amsmath,amssymb,amsthm,graphicx,float,xcolor,geometry,cite,enumitem,fancyhdr,lipsum,subfig,sidecap}

\newtheorem{theorem}{Theorem}

\RequirePackage[pdfusetitle]{hyperref}
\hypersetup{hidelinks,colorlinks,urlcolor=[rgb]{0,0,1},linkcolor=[rgb]{0,0,0},citecolor=[rgb]{0,0,0}}

\makeatletter

\setlist[enumerate]{topsep=1ex,itemsep=-0.5ex,leftmargin=*}
\setlist[enumerate,2]{topsep=-0.5ex,itemsep=-0.5ex,leftmargin=*}
\setlist[itemize]{topsep=1ex,itemsep=-0.5ex,leftmargin=*}
\setlist[itemize,2]{topsep=-0.5ex,itemsep=-0.5ex}

\title{Rigid foldability of the augmented square twist }
\author{Thomas C. Hull\thanks{Western New England University, {\tt \{thull, michael.urbanski1\}@wne.edu}} \and Michael T. Urbanski\footnotemark[1]}
\date{}

\makeatletter
  \let\runtitle\@title
  \let\runauthor\shortauthor
\makeatother

\begin{document}

\maketitle

\begin{abstract}
  Define the augmented square twist to be the square twist crease pattern with one crease added along a diagonal of the twisted square.  In this paper we fully describe the rigid foldability of this new crease pattern.  Specifically, the extra crease allows the square twist to rigidly fold in ways the original cannot.  We prove that there are exactly four non-degenerate rigid foldings of this crease pattern from the unfolded state.
\end{abstract}

\section{Introduction}
\label{sec:introduction}
The classic square twist, shown in Figure~\ref{HU-fig0}(a) where bold creases are mountains and non-bolds are valleys, is well-known in origami art from its use in the Kawasaki rose \cite{conn} and in origami tessellations \cite{Gjerdre}.  However, this is also an example of an origami crease pattern that is not {\em rigidly-foldable}, meaning that it cannot be folded along the creases into the flat, square-twist state without bending the faces of the crease pattern along the way.  (See \cite{projectorigami} for one proof of this.)  In fact, this lack of rigid foldability has been used by researchers to study the mechanics of bistability in origami \cite{Jesse}.  Interestingly, there are ways one can change the mountain-valley assignment of the square twist to make it rigidly foldable; they require the inner diamond of Figure~\ref{HU-fig0}(a) to be either MMVV, MVVV, or VMMM \cite{Langtwists}.

\begin{figure}[h]
  \centering
  \includegraphics[scale=.5]{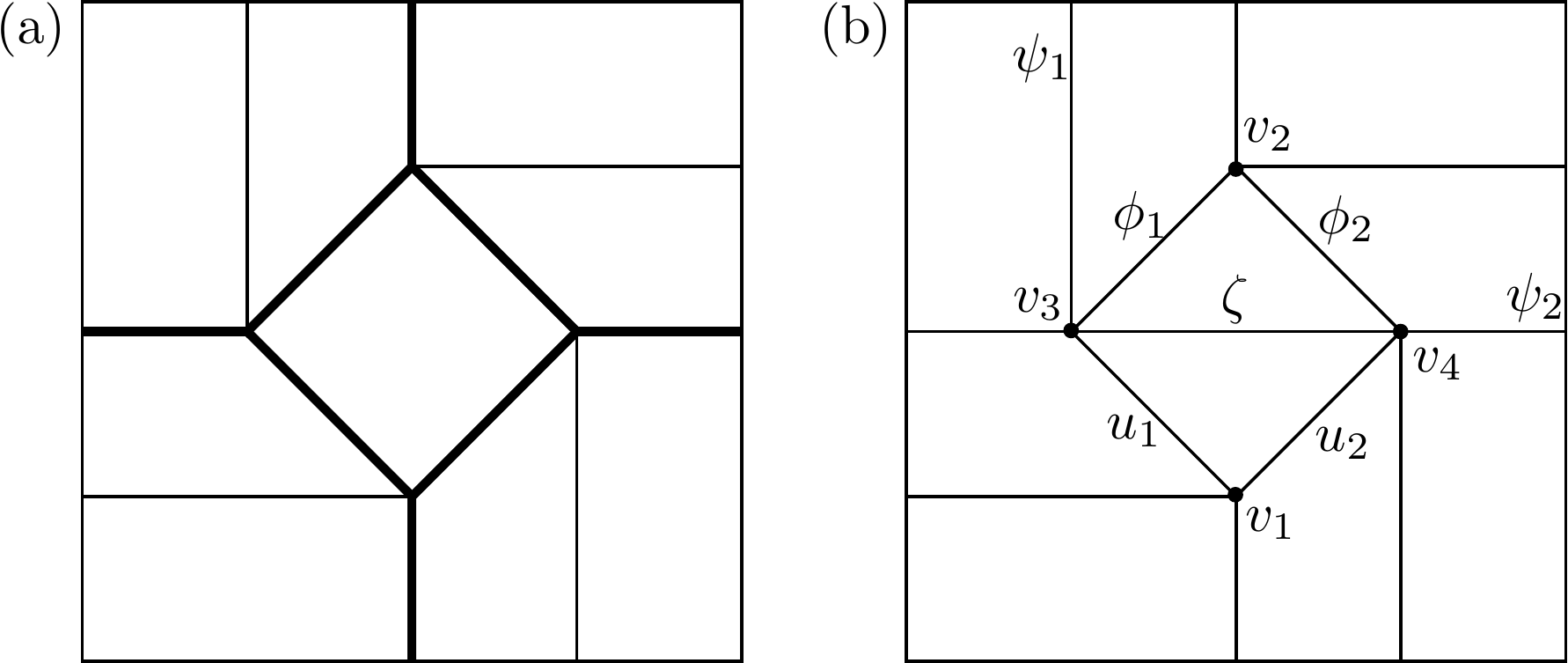}
  \caption{(a) The classic square twist with mountain-valley assignment.  (b) The augmented square twist, with vertices and folding angles labeled.}
  \label{HU-fig0}
\end{figure}

A logical question to ask is: Could we add creases to the classic square twist to make it rigidly foldable?  The answer is, ``Yes!'' and the simplest way is to add a single crease along a diagonal of the inner square diamond, as shown in Figure~\ref{HU-fig0}(b).  In this paper we will prove the rigid-foldability of this augmented square twist crease pattern and characterize its configuration space (i.e., its kinematics).

\section{Preliminaries}
\label{sec:section1}
There are a number of tools that we will be using in this paper.

\begin{figure}
  \centering
  \includegraphics[scale=.4]{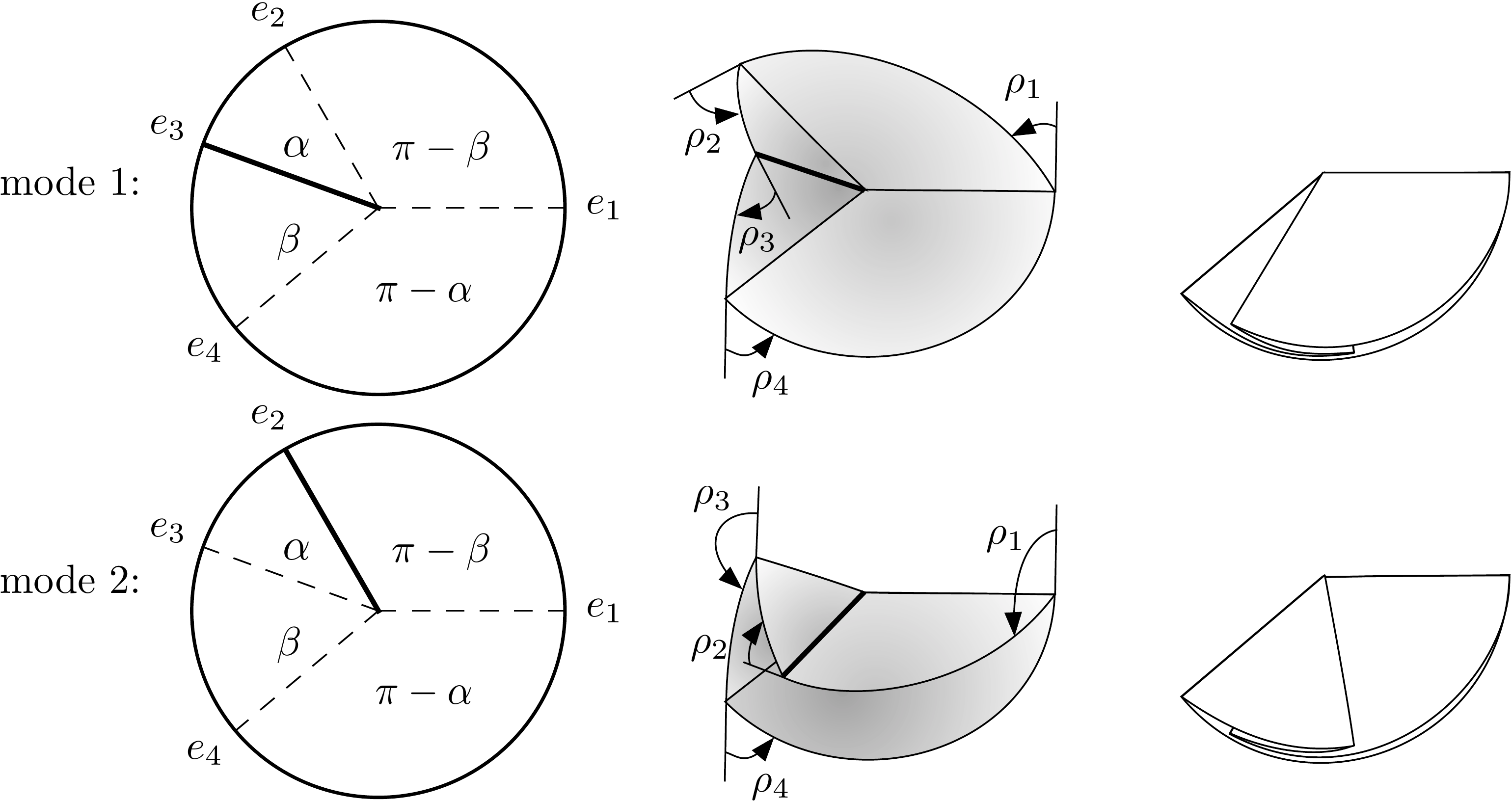}
  \caption{The two modes of a rigid folding of a degree-4 flat-foldable vertex.}
  \label{HU-fig1}
\end{figure}

One is the set of folding angle equations for a degree-4, flat-foldable vertex to rigidly fold and unfold \cite{Langtwists,Tachi2016}.  If we let $\alpha$ and $\beta$ be consecutive plane angles between the creases with $\alpha<\beta$ and $\alpha+\beta<\pi$, and we let $\rho_1,\ldots, \rho_4$ be the folding angles as illustrated in Figure~\ref{HU-fig1}, then there are equations for the mode 1 and mode 2 folding angles: 
$$\rho_1=-\rho_3,\ \rho_2=\rho_4\mbox{ in mode 1, }\rho_1=\rho_3,\ \rho_2=-\rho_4\mbox{ in mode 2,}$$
$$\tan\frac{\rho_1}{2} = \frac{\cos \frac{\alpha+\beta}{2}}{\cos \frac{\alpha-\beta}{2}}\tan\frac{\rho_2}{2} \mbox{, and }
\tan\frac{\rho_2}{2} = \frac{\sin \frac{\alpha-\beta}{2}}{\sin \frac{\alpha+\beta}{2}}\tan\frac{\rho_1}{2}.$$
For the degree-4 vertices in the augmented square twist, we have $\alpha=45^\circ$ and $\beta=90^\circ$, which gives us
$$\frac{\cos \frac{135^\circ}{2}}{\cos\frac{-45^\circ}{2}} = -\frac{\sin \frac{-45^\circ}{2}}{\sin\frac{135^\circ}{2}} = \tan\frac{45^\circ}{2} = \sqrt{2}-1.$$
Thus we get, for mode 1 and mode 2 respectively,
$$\rho_1=2\arctan((\sqrt{2}-1)\tan\frac{\rho_2}{2})\mbox{ and }\rho_2 = -2\arctan((\sqrt{2}-1)\tan\frac{\rho_1}{2}).$$

The other major tool we will use is Balkcom's procedure for deriving the kinematic equations for rigidly folding a vertex of any degree \cite{Balkcom}.  We'll apply this to the degree-5 vertices in the augmented square twist as follows:  Impose a coordinate system where the unfolded paper is in the $xy$-plane in $\mathbb{R}^2$, $v_3$ is the origin, and the crease whose folding angle is $\phi_1$ lies along the positive $x$-axis, as in Figure~\ref{HU-fig2.5}.

\begin{figure}
  \centering
  \includegraphics[scale=.6]{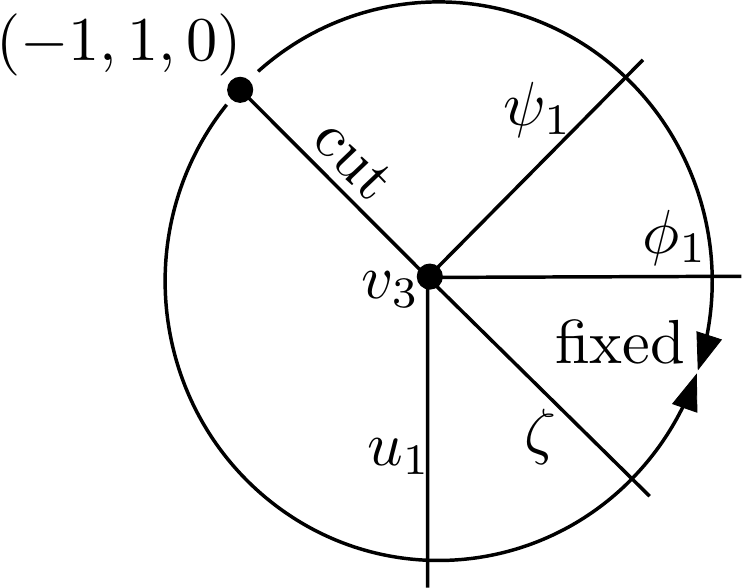}
  \caption{Analyzing the folding angles of one of the degree-5 vertices.}
  \label{HU-fig2.5}
\end{figure}

We cut the paper from $v_3$ along the crease to the point $(-1,1,0)$ and then imagine folding the two ends of this cut towards the fixed sector of paper between the $\phi_1$ and $\zeta$ creases; see Figure~\ref{HU-fig2.5}. The image of the point $(-1,1,0)$ under these two folds must be equal.  Therefore,  if we let $R_x(\theta)$ and $R_z(\theta)$ denote the $3\times 3$ matrices that rotate $\mathbb{R}^3$ by angle $\theta$ about the $x$- and $z$-axes, respectively, we must have
\begin{equation}\label{eqBalk}
\begin{split}
R_z(-45^\circ)R_x(-\zeta)R_z(45^\circ)R_z(-90^\circ)R_x(-u_1)R_z(90^\circ)
(-1,0,1)^T
= \\
R_x(\phi_1)R_z(45^\circ)R_x(\psi_1)R_z(-45^\circ)(-1,1,0)^T.
\end{split}
\end{equation}

We let the folding angles $u_1$ and $\zeta$ be the independent variables (since every degree-5 vertex will have 2 degrees of freedom) and let $\phi_1$ and $\psi_1$ be the dependent variables.  Then the $x$-coordinate of the identity in Equation~\eqref{eqBalk} gives us a formula for $\psi_1$ in terms of $u_1$ and $\zeta$ (note that $u_1,\zeta\in(-\pi,\pi))$:
\begin{equation}\label{eq0}
\cos\psi_1 = \frac{1}{2}(1-\cos\zeta + \cos u_1(1+\cos\zeta) - \sqrt{2}\sin u_1 \sin\zeta).
\end{equation}
Since $\psi_1\in(-\pi,\pi)$, there will be two solutions for $\psi_1$ of the form $\pm a\in(-\pi,\pi)$; we will denote these $\psi_1^+$ and $\psi_1^-$.  Each of these solutions will then determine a solution for $\phi_1$ from the other coordinates of Equation~\eqref{eqBalk}, and we will denote these $\phi_1^+$ and $\phi_1^-$.  (These equations are awkward and unenlightening, but their graphs will be shown in subsequent Figures.)

For the degree-5 vertex $v_4$ we will have $u_2$ and $\zeta$ be the independent variables and $\phi_2$ and $\psi_2$ the dependent ones. Again, the two solutions for $\psi_2$ will determine a pair of solutions for $\phi_2$, which we denote $\phi_2^+$ and $\phi_2^-$.


\begin{theorem}\label{HU-thm1}
The augmented square twist crease pattern, when rigidly folded, is a 1-degree-of-freedom (DOF) system.
\end{theorem}

\begin{proof} 
Let $u_1$ be a free parameter.  Then $u_i=(-1)^i 2\arctan((\sqrt{2}-1)\tan(u_j/2))$ by the folding angle equations for $v_1$.  (Here $(i,j) = (2,1)$ for mode 1 and $(1,2)$ for mode 2.)  Then allowing, for the moment, $\zeta$ to be freely chosen, we can determine a finite number of possible folding angles for $\phi_1, \psi_1, \phi_2$, and $\psi_2$, each of which will be functions of $u_1$ and $\zeta$.  Now, since the folding angles $\phi_1$ and $\phi_2$ meet at the degree-4 vertex $v_2$, we have 
\begin{equation}\label{eq1}
\phi_i=(-1)^j 2\arctan((\sqrt{2}-1)\tan(\phi_j/2)),
\end{equation}
where $(i,j)=(1,2)$ for mode 1 and $(2,1)$ for mode 2.  We can solve Equation \eqref{eq1} for $\zeta$ and thus make $\zeta$ a dependent variable.  That is, the {\em loop constraint} of the square twist cycle $\{v_1,v_3,v_2,v_4\}$ removes one of the degrees of freedom from the degree-5 vertices. Thus this rigidly-foldable crease pattern is a 1-DOF system. 
\end{proof}

\noindent{\bf Remark 1:}
Note that there are some slight exceptions to Theorem~\ref{HU-thm1}.  For example, at the unfolded state there are multiple combinations of folding angle velocities that can be chosen to fold into different three-dimensional shapes.  As we will see in the next section, we can think of the space of folded configurations as a manifold, and points such as the origin (representing the unfolded state) will be singular points of this manifold.  Such singular points have multiple, but a finite number of, tangent vectors via which the crease pattern may fold.  Since these singular points do not define a fully two- (or higher) dimensional tangent space, they are not considered 2-DOF (or higher) points of the origami mechanism.

\section{Counting modes}\label{sec:section2}

The {\em configuration space} of a rigidly folding origami crease pattern with $n$ crease lines is a subset of $C\subset \mathbb{R}^n$ where for every point $p\in C$, the coordinates of $p$ equal the folding angles of the creases for one of the valid rigid folding states of the crease pattern.  That is, if our creases are labeled $1, \ldots, n$ and their folding angles are $\rho_1,\ldots, \rho_n$, then $p=(\rho_1,\ldots, \rho_n)$ would be a point in $C$.  As the creases flex, the points $p$ will trace out a manifold in $\mathbb{R}^n$ whose intrinsic dimension will depend on the degrees of freedom of the crease pattern.

Since the augmented square twist crease pattern is a 1-DOF system, its configuration space will consist of a collection of curves.  Each curve will be a folding mode, and we want to count how many of these curves will pass through the origin (the unfolded state).  

We also only want folding modes that have all four vertices folding.  That is, we could fold the paper in half, only using the horizontal creases passing through $v_3$ and $v_4$ being folded and the rest unfolded.  That will make a (straight line) curve through the origin, but it doesn't account for a very interesting way to fold this crease pattern.  Rather, it is a {\em degenerate} case, since it's not using all of the creases in the pattern.

So, if all four vertices must fold, then the degree-4 vertices $v_1$ and $v_2$ can fold in either mode 1 or mode 2. For each combination, we have choices to make for the degree-5 vertices $v_3$ and $v_4$.  The number of such choices will depend on the number of intersections of the curves $\phi_1^\pm$ and $\phi_2^\pm$ when made to satisfy the loop constraint.  We enumerate these cases below.

\subsection{Case 1: $v_1$ and $v_2$ are both mode 1}

Then we have, from the requirements of the degree-4 vertices,
\begin{equation*}
    \tan\frac{u_2}{2}   = (\sqrt{2}-1)\tan\frac{u_1}{2}\mbox{ and } 
    \tan\frac{\phi_1}{2}   = (\sqrt{2}-1)\tan\frac{\phi_2}{2}.
  \label{eq3}  
\end{equation*}
 Thus
$u_2=2\arctan((\sqrt{2}-1) \tan(u_1/2))$ and $\phi_1=2\arctan((\sqrt{2}-1)\tan(\phi_2/2))$. 
 Ignoring for the moment the loop constraint of the square twist, the kinematic equations for $\phi_1$ and $\phi_2$ will be functions of $u_1$ and $\zeta$.  Figure~\ref{HU-fig3} shows, for a fixed value of $u_1=1.6$, the graphs of $y=\phi_1^+$ and $y=\phi_1^-$ with respect to $\zeta$, which form a closed curve. Alongside of this are the graphs of $y=2\arctan((\sqrt{2}-1)\tan(\phi_2^+/2))$ and the same for $\phi_2^-$, which form a closed curve that ``wraps around" vertically (which makes sense if we imagine a folding angle of $+\pi$ being equivalent to one of $-\pi$).  Where these two curves intersect give $\zeta$-values where the loop constraint of the square twist diamond will be satisfied (the kinematic equations of the two degree-5 vertices will also satisfy the degree-4 vertex folding angles).

\begin{figure}
  \centering
   \includegraphics[width=\linewidth]{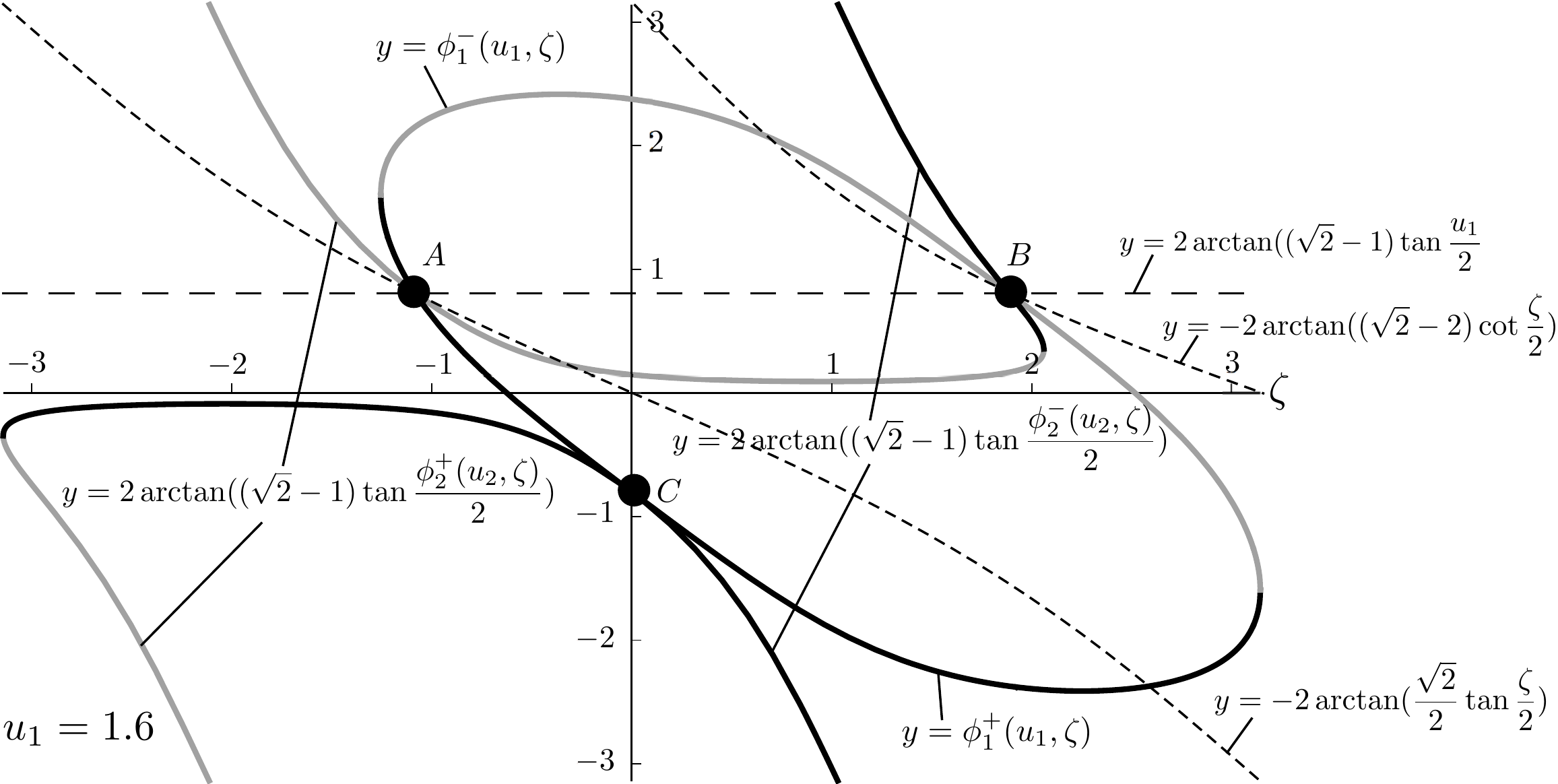}
  \caption{Plots of the curves $y=\phi_1^\pm$ and $y=2\arctan((\sqrt{2}-1)\tan(\phi_2^\pm/2))$ in case 1 with respect to the folding angle $\zeta$, with their intersection points labeled $A,B,C$, when $u_1=1.6$.}
  \label{HU-fig3}
\end{figure}

We labeled the three points of intersection of the $\phi_1^{\pm}$ and $2\arctan((\sqrt{2}-1)$ $\tan(\phi_2^{\pm}/2))$ curves by $A$, $B$, and $C$ in the above figure.  The equations for these curves are rather cumbersome, but algebraic manipulations on software like {\em Mathematica} show that for $-\pi\leq u_1\leq \pi$ we have the following:

\begin{itemize}
\item The point $A$ travels along the curve $y=-2\arctan((\sqrt{2}/2)\tan(\zeta/2))$.
\item The point $B$ travels along the curve $y=2\arctan((\sqrt{2}-2)\cot(\zeta/2))$.
\item The point $C$ always remains on the $\zeta=0$ axis.
\item The points $A$ and $B$ are both on the horizontal line $y=2\arctan((\sqrt{2}-1)\tan(u_1/2))$, which is the folding angle $u_2$ when $v_1$ is in mode 1.
\end{itemize}

Therefore, we have that the intersection point $A$ gives a folding mode of the crease pattern where we have $\phi_1 = u_2$.  This folding, illustrated in Figure~\ref{HU-fig4}, results in the standard way origamists fold a square twist, with all mountain creases around the central square diamond. Note that in the folding shown in Figure~\ref{HU-fig4} an extra move is added at the end to make it open up into the classic square twist.

\begin{figure}
  \centering
   \includegraphics[scale=.25]{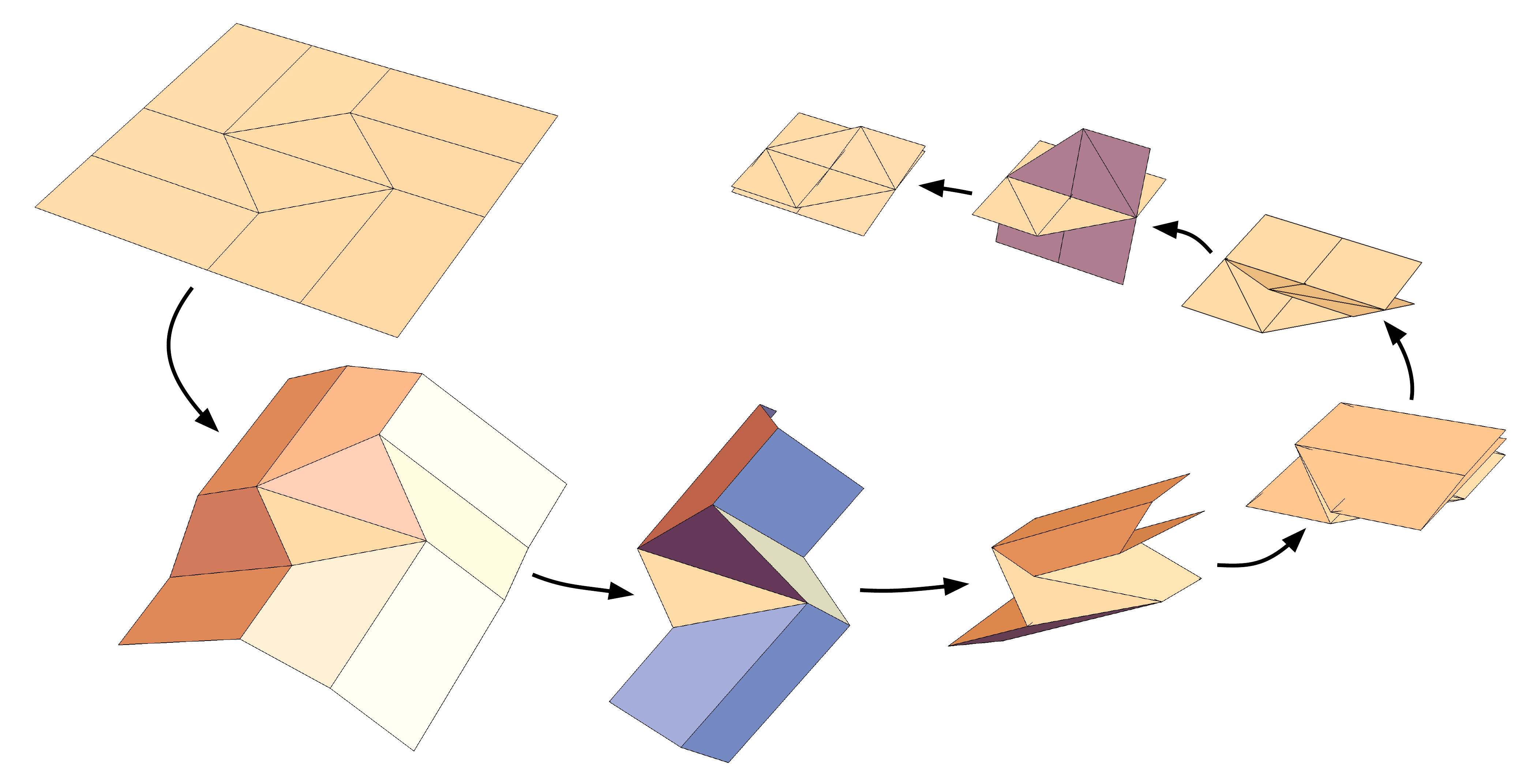}
  \caption{The rigid folding of the augmented square twist along the curve traced by $A$ in Figure~\ref{HU-fig3}.}
  \label{HU-fig4}
\end{figure}

For this folding mode, we have the interesting fact that the relationship between the folding angles $u_1$ and $\zeta$ is linear in the tangent of the half angles: 
$$\tan\frac{\zeta}{2} = (\sqrt{2} - 2)\tan\frac{u_1}{2}.$$

The folding mode indicated by the intersection point $C$ is degenerate.  The path traveled by $C$ shows that the center crease with folding angle $\zeta=0$ is never folded.  Thus in this case the folding angle equations merely provide one of the rigid foldings of the square twist, specifically the one shown in the Figure~\ref{HU-fig5}(a).

\begin{figure}
  \centering
   \includegraphics[scale=.3]{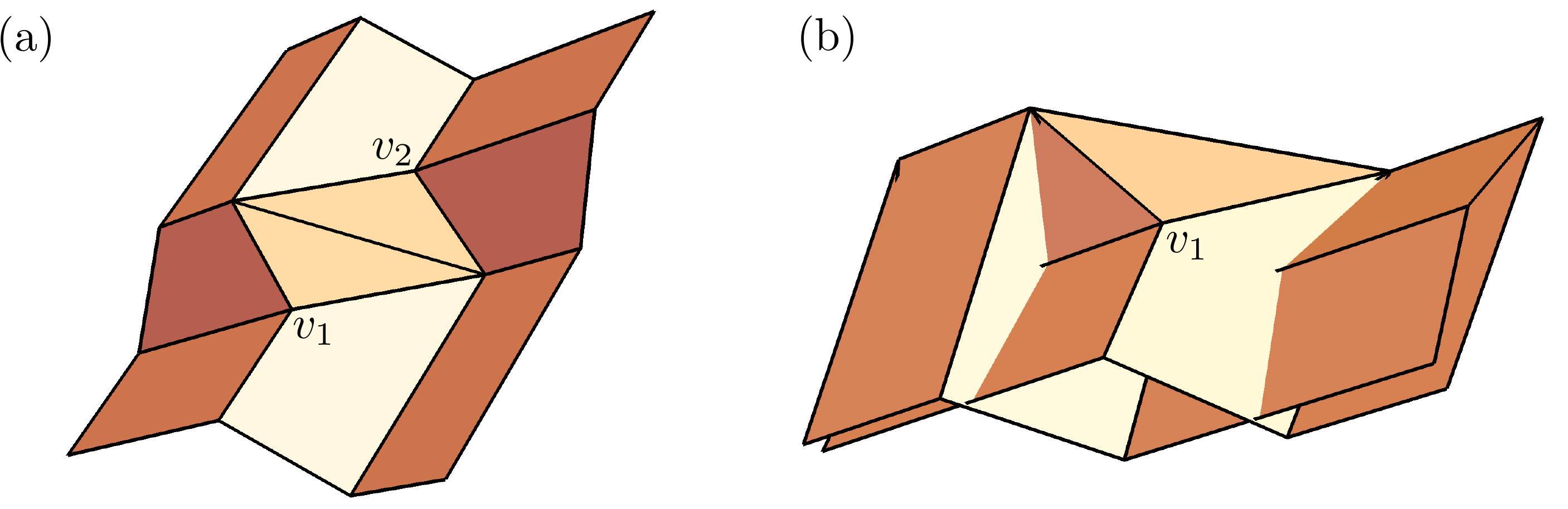}
  \caption{(a) The degenerate rigid folding of the trace of point $C$, which is a MMVV rigid folding of the non-augmented square twist. (b) A still image of the rigid folding made by the trace of $B$ in Figure~\ref{HU-fig3}, which self-intersects and never unfolds completely.}
  \label{HU-fig5}
\end{figure}

In Figure~\ref{HU-fig5}(b) we see an image of the folding mode that point $B$ traces.  Note that the curve made by the intersection point $B$ does not pass through the origin.  This means that this folding mode does not include the flat, unfolded state.  In fact, as you can see in Figure~\ref{HU-fig5}(b), this folding mode has the entire horizontal crease folded as if it was one crease.  This is not a legitimate folding mode since it doesn't include the unfolded state, not to mention requiring that the paper self-intersect.

\subsection{Case 2: $v_1$ and $v_2$ are both mode 2}

This case is handled similarly to Case 1.  Figure~\ref{HU-fig6} shows the oval curve made by $\phi_1^\pm$ that is the same as in Case 1.  This time, however, we want to see where this curve intersects those given by $\phi_2^\pm$ when combined with $v_1$ and $v_2$ being mode 2.  As such, we need 

\begin{figure}
  \centering
   \includegraphics[width=\linewidth]{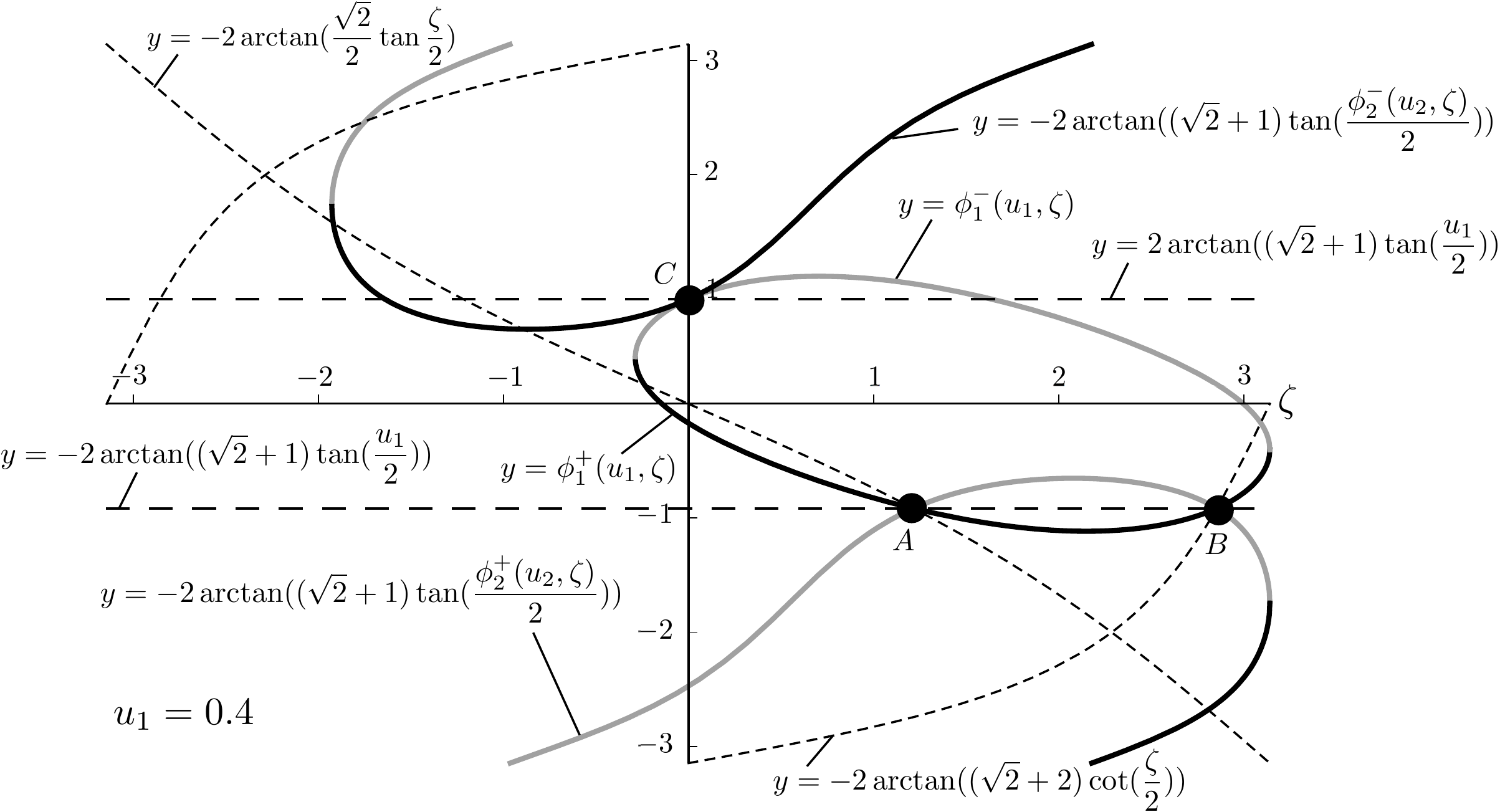}
  \caption{Plots of the case 2 $\phi_1^\pm$ and $-2\arctan((\sqrt{2}+1)\tan(\phi_2^\pm/2))$ curves with respect to the folding angle $\zeta$, with their intersection points labeled $A,B,C$, when $u_1=0.4$.}
  \label{HU-fig6}
\end{figure}

\begin{equation*}
    \tan\frac{u_1}{2}  = -(\sqrt{2}-1)\tan\frac{u_2}{2}\mbox{ and } \\
    \tan\frac{\phi_2}{2}  = -(\sqrt{2}-1)\tan\frac{\phi_1}{2}.
  \label{eq4}  
\end{equation*}
This means that we need to substitute $u_2=-2\arctan((\sqrt{2}+1)\tan(u_1/2))$ (since $1/(\sqrt{2}-1) = \sqrt{2}+1$) into the $\phi_2^\pm(u_2,\zeta)$ formulas and then plot $y=-2\arctan((\sqrt{2}+1)\tan(\phi_2^\pm(u_2,\zeta)))$.  This gives us, again, three intersection points with $y=\phi_1^\pm(u_1,\zeta)$; these are labeled $A$, $B$, and $C$ in Figure~\ref{HU-fig6} (where we picked $u_1=0.4$).

\begin{figure}
  \centering
   \includegraphics[width=\linewidth]{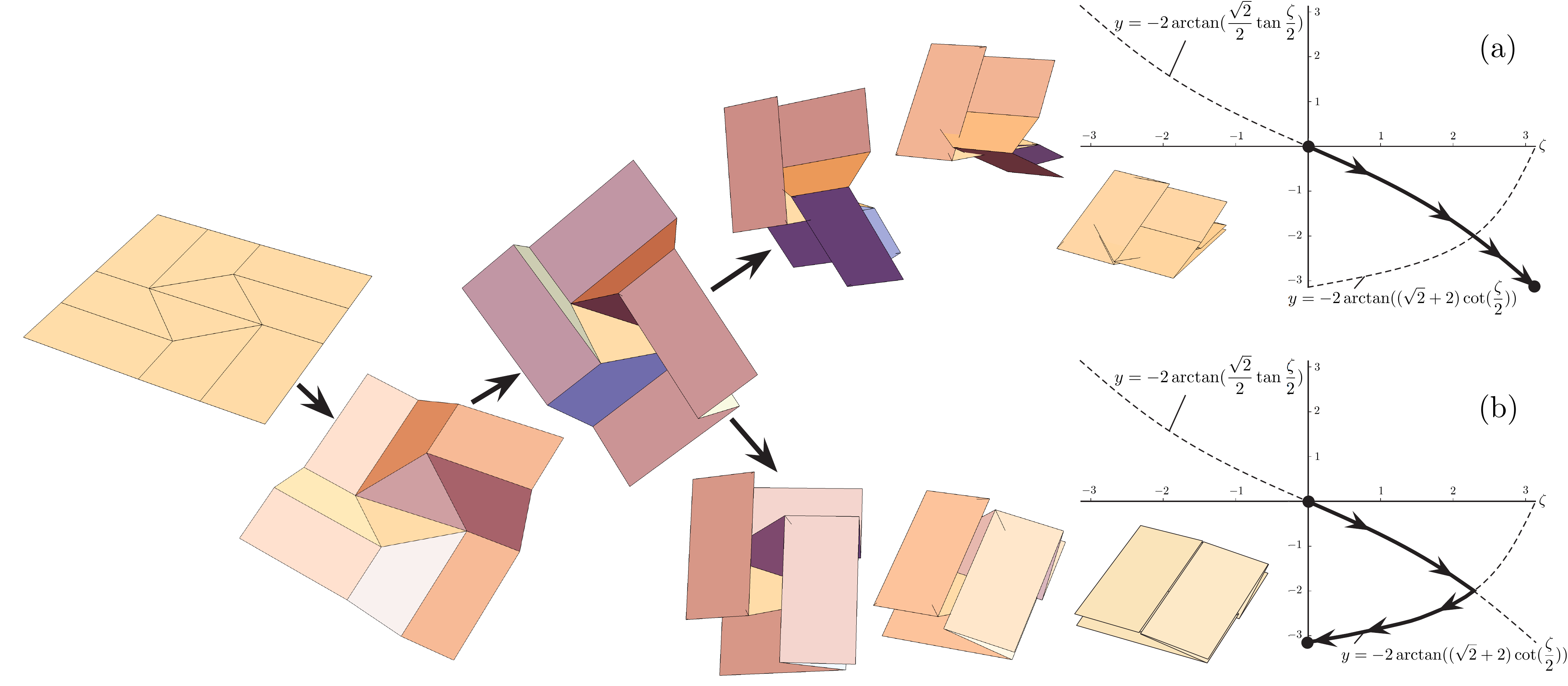}
  \caption{The augmented square twist folding along (a) the curve traced by $A$ in Figure~\ref{HU-fig6} and (b) a combination of the $A$ and $B$ curves.}
  \label{HU-fig7}
\end{figure}

Manipulating the equations and intersection points in Figure~\ref{HU-fig6} reveals the following for $-\pi\leq u_1\leq\pi$:

\begin{itemize}
\item The point $A$ travels along the curve $y=-2\arctan((\sqrt{2}/2)\tan(\zeta/2)).$
\item The point $B$ travels along the curve $y=-2\arctan((\sqrt{2}+2)\cot(\zeta/2)).$
\item The point $C$ always remains on the $\zeta=0$ axis.
\item The points $A$ and $B$ are both on the horizontal line $y=2\arctan((\sqrt{2}+1)\tan(u_1/2))$, which is the folding angle $u_2$.
\end{itemize}

Thus, as in Case 1, the trace of point $C$ indicates a degenerate folding mode where the augmented crease is not folded ($\zeta=0$), giving us one of the rigid foldings of a square twist.  The trace of point $B$ indicates a rigid folding that does not intersect the origin, and thus never unfolds completely.

The intersection point $A$ describes a rigid folding of a well-known (to origamists) way to fold a square twist, namely the iso-area square twist, popularized in the ``stretch wall'' model by Yoshihide Momotani \cite{Momotani}. Again, we obtain an interesting tangent-of-half-angles relationship between $\zeta$ and $u_1$:
$$\tan\frac{\zeta}{2}=(\sqrt{2}+2)\tan\frac{u_1}{2}.$$
However, the trace of $A$ from, say, the origin to $(\pi,-\pi)$ in Figure~\ref{HU-fig6} is not the way this twist is typically folded in that it folds the square diamond of the twist in half; see the folding illustration and graph marked (a) in Figure~\ref{HU-fig7}.  Rather, when origamists fold this version of the square twist, the square diamond is allowed to remain open, which would require $\zeta$ to become zero again once the other creases in the twist were folded.  Interestingly, the equations in our analysis do capture this folding and thus prove that it can be done rigidly with with the extra $\zeta$ crease.  This is shown in Figure~\ref{HU-fig7}(b), where we trace the point $A$ away from the origin until it intersects the curve traversed by $B$, whereupon we travel along the $B$ curve to the point $(0,-\pi)$.  

\subsection{Case 3: $v_1$ is mode 1 and $v_2$ is mode 2 (or vice-versa)}

\begin{figure}
  \centering
   \includegraphics[width=\linewidth]{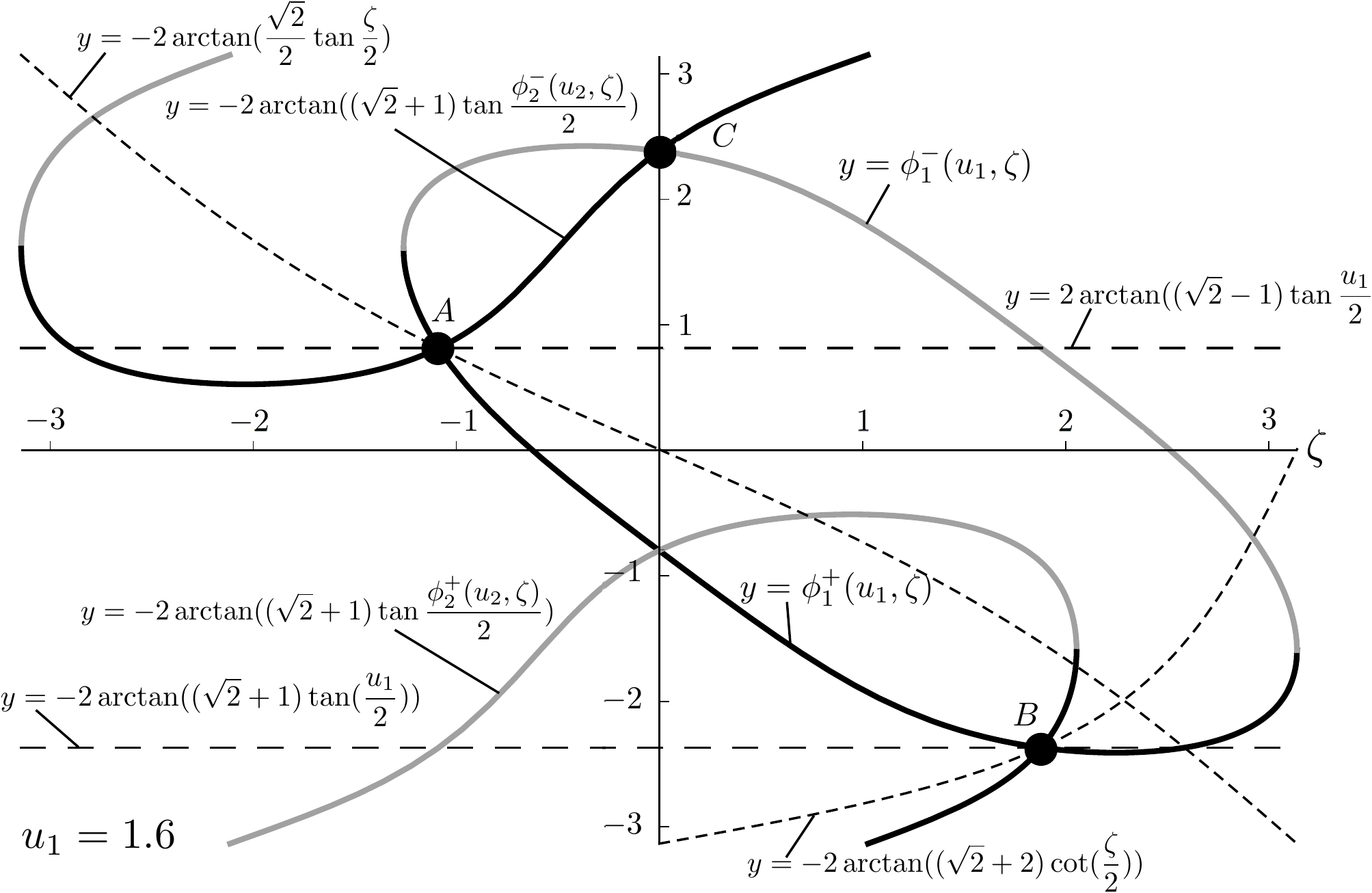}
  \caption{Plots of the case 3 $\phi_1^\pm$ and $-2\arctan((\sqrt{2}+1)\tan(\phi_2^\pm/2))$ curves with respect to the folding angle $\zeta$, with their intersection points labeled $A,B,C$, when $u_1=1.6$.}
  \label{HU-fig8}
\end{figure}

There are two other non-degenerate rigid folding modes of the augmented square twist.  They are the cases where vertices $v_1$ and $v_2$ have different degree-4 folding modes.  The folding angle curve analysis for the case where $v_1$ is mode 1 and $v_2$ is mode 2 is shown in Figure~\ref{HU-fig8}.  We still have three intersection points $A$, $B$, $C$ between the $y=\phi_1^\pm$ curves and the $\phi_2^\pm$ curve, modified so that vertex $v_2$ is in mode 2.  However, for these $\phi_2^\pm$ curves we need to substitute $u_2=2\arctan((\sqrt{2}-1)\tan(u_1/2))$ to make sure vertex $v_1$ is in mode 1.  

Again, the intersection point $A$ traces the curve $y=-2\arctan((\sqrt{2}/2)\tan(\zeta/2))$ and indicates a legitimate folding mode, while the point $B$ traces a curve that doesn't pass through the origin, and point $C$ remains on the $\zeta=0$ axis (and thus is a degenerate folding mode).  

A picture similar to that in Figure~\ref{HU-fig8} results in the case where $v_1$ is mode 2 and $v_2$ is mode 1; the rigid folding for the point $A$ in this case is shown in Figure~\ref{HU-fig9}.

\begin{figure}
  \centering
   \includegraphics[width=\linewidth]{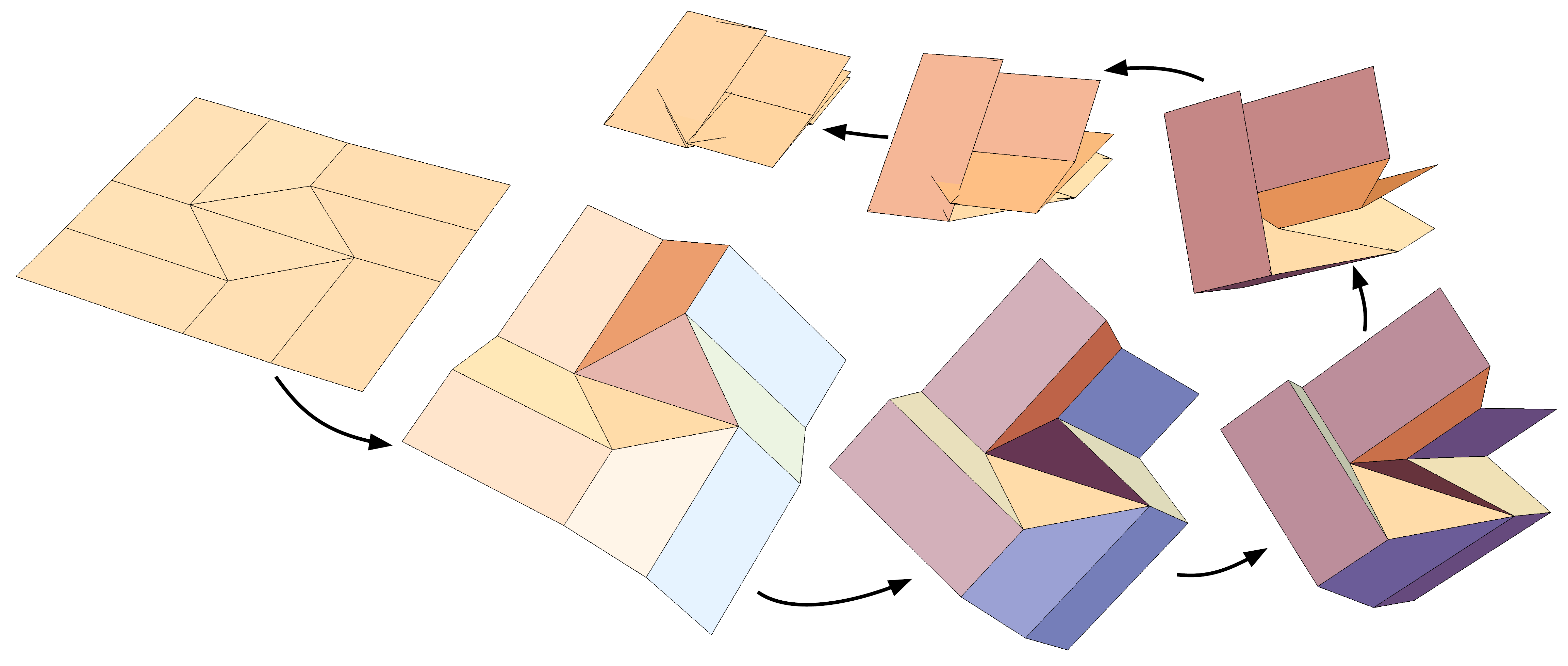}
  \caption{The rigid folding in the case where vertex $v_1$ is mode 2 and $v_2$ is mode 1.}
  \label{HU-fig9}
\end{figure}

\vspace{.1in}

\noindent{\bf Remark 2:}  In all of the rigid foldings of the augmented square twist, the folding angle equations model the paper folding flat, where all the folding angles either $\pm\pi$ or zero.  For example, the final folded forms in Figures~\ref{HU-fig4} and \ref{HU-fig7}(b) have $\pm\pi$ folding angles at all the creases except for $\zeta$, whose folding angle is zero.  One could follow the stacking order of the crease pattern faces as the folding angles approach $\pm\pi$ in order to obtain a layer ordering of the paper in the final, flat-folded form.  

However, this is a misleading way to think about layer ordering in flat origami models.  To see this, consider the classic square twist as folded in Figure~\ref{HU-fig4}.  If one makes a physical model of this square twist, one quickly sees that the rectangular faces of the crease pattern interweave.  That is, if we labeled the rectangular faces of the augmented square twist crease pattern $F_1, \ldots, F_4$ and tried to impose a layer ordering on them based on the flat-folded form, we would have
$$F_1 < F_2 < F_3 < F_4 < F_1,$$
which is not a valid order relation.  In order to make such layer orderings meaningful, one needs to take an approach such as that suggested by Justin \cite{Justin}, where faces of the crease pattern are further subdivided so as to avoid problems with the order relation.  

\section{Conclusions}

We have seen that while the standard square twist (where the central diamond has MMMM or VVVV mountain-valley assignment) and the iso-area square twist (where the central diamond is MVMV) are not rigidly-foldable, they do become rigidly foldable when an extra crease is added along one diagonal of the central diamond.  In proving this, we discovered that there are only four non-degenerate folding modes of this augmented square twist crease pattern that are connected to the unfolded state.  That is, there are four curves that pass through the origin in the configuration space of this crease pattern that are not constantly zero on the $\zeta$-axis.  There are 6 curves that also pass through the origin and have $\zeta=0$ (i.e., lie in a hyperplane that is orthogonal to the $\zeta$-axis), but these are degenerate cases, representing ways to rigidly fold the (non-augmented) square twist crease pattern.  Our work derives the equations for the folding angles in these cases, and thus the equations of the curves in the configuration space.  
Interested readers may access the {\em Mathematica} code we used for these equations, as well as the graphics shown in this paper, here: \url{http://mars.wne.edu/~thull/ast/ast.html}.
These equations could be used to analyze the self-foldability of any of these folding modes, under the definition of self-foldability from \cite{Tachi2016}.

While degree-4 vertices are well-studied, not much has been written about rigid foldings of degree-5 vertices.  Using Balkcom's method \cite{Balkcom}, we were able to fully study the kinematics of the degree-5 vertices in the augmented square twist and thus characterize its configuration space.  Perhaps the methods used here can be extended to other rigidly-foldable crease patterns with degree-5 and degree-4 vertices.  In doing so, the library of rigid origami mechanisms whose kinematics are completely understood would grow considerably.

\bibliographystyle{osmebibstyle}
\bibliography{7OSME-augsqtwist}



\end{document}